\documentclass{article}

\usepackage{amsmath}
\usepackage{mathrsfs}
\usepackage{amsfonts}
\usepackage{amsthm}
\usepackage{amssymb,color}
\usepackage{graphicx}
\usepackage[square,numbers]{natbib}
\usepackage{verbatim}
\usepackage{enumerate}
\usepackage{tikz}
\usetikzlibrary{patterns}
\usepackage{float}
\usepackage{nicefrac}
\usepackage{hyperref}
\usepackage{setspace}
\usepackage[pagewise]{lineno}

\newtheorem{theorem}{Theorem}
\newtheorem{corollary}[theorem]{Corollary}
\newtheorem{lemma}[theorem]{Lemma}
\newtheorem{proposition}[theorem]{Proposition}

\theoremstyle{definition}
\newtheorem{definition}[theorem]{Definition}

\theoremstyle{remark}
\newtheorem{remark}[theorem]{Remark}
\newtheorem{claim}{Claim}

\numberwithin{theorem}{section}
\numberwithin{equation}{section}

\newcommand{\B}[1]{\mathbb{#1}}
\newcommand{\C}[1]{\mathcal{#1}}

\DeclareMathOperator*{\lip}{Lip}

\begin{document}

\title{An Overview on Laakso Spaces}

\author{Marco Capolli\\[1em] \small Università degli studi di Padova, Department of Mathematics ``Tullio Levi-Civita''}

\date{\today}

\maketitle
 
\begin{abstract}
    Laakso's construction is a famous example of an Ahlfors $Q$-regular metric measure space admitting a weak $(1,1)$-Poincar\'{e} inequality that can not be embedded in $\B{R}^n$ for any $n$. The construction is of particular interest because it works for any fixed dimension $Q>1$, even fractional ones. In this paper we will shed some light on Laakso's work by expanding some of his statements and proving results that were left unproved in the original paper.
    \bigskip

    \smallskip

\end{abstract}
	
\section{Introduction}

\renewcommand{\thefootnote}{\Alph{footnote}}
	
    The doubling property and the Poincar\'e inequality are two important tools used in mathematical analysis and, more specifically, in geometric measure theory. They are regularity properties that are often required in order to work in a setting with controlled geometry. In this context, the term \emph{PI space} is used to identify those spaces that admit both properties. 
	
    There are many interesting directions in which the study of PI spaces can go. One is to study properties of differentiable functions in metric measure space \cite{alberti2016differentiability, preiss2012frechet, P}. Worth mentioning in this context is a work by Cheeger \cite{cheeger1999differentiability}, in which he proved that the assumption of being a PI space is enough to recover a Rademacher-like theorem for differentiability of Lipschitz functions. The reader curious about this argument can also find many interesting results in this sense in \cite{bate2015structure} and the references therein.
	
    Related to Rademacher's theorem, much research has been done in the opposite direction. The question asked in this case is the following: given a null set $N\subset X$ in a metric measure space, is it possible to find a Lipschitz function $f:X\to\B{R}$ that is differentiable at no point of $N$? The answer to this question, which is positive in some cases and negative in others, lead to the definition of \emph{Universal Differentiability Sets} (UDS for short). A UDS is a set such that, for every real-valued Lipschitz function, at least one point of the set is a point of differentiability for the function. Their existence has been studied first in the Euclidean setting \cite{dore2011compact, preiss2015differentiability} and then in more general metric measure spaces \cite{pinamonti2017measure, pinamonti2020universal}. 

    The reader interested in other results concerning differentiability of Lipschitz function and the structure of null sets may also refer to \cite{alberti2010differentiability} and the references therein.
	
    Many of the techniques used to construct a UDS come from a work by Preiss \cite{preiss1990differentiability}. The main tool is a relation between the concept of differentiability and the existence of a \emph{maximal directional derivative}. So far, in the literature, this approach has always been applied to PI spaces. 
	
    In \cite{capolli2022maximal}, M. Capolli, A. Pinamonti and G. Speight, investigated to what extent the techniques from \cite{preiss1990differentiability} can be extended to the setting of Laakso spaces. The interest in the question relies on the fact that, as proved in \cite{laakso2000ahlfors}, Laakso spaces are PI spaces. However, they do not posses many of the properties that Euclidean spaces or Carnot groups have. For example, Laakso spaces do not have an underlying group structure. Due to this fact it is not possible to define translations and dilations in a Laakso spaces. As showed in \cite{capolli2022maximal} the absence of a linear structure makes working in this setting more difficult and the techniques from \cite{preiss1990differentiability} cannot be applied light-heartedly in this context.
	
    All the proofs and techniques introduced in \cite{capolli2022maximal} are based on definitions and results from \cite{laakso2000ahlfors}. However, the original work lacks of some explanations and rigorous proofs for most of the results presented in the first chapter. Indeed, the proof of many of the results presented by Laakso in the first part of his work were left to the intuition of the reader. It is the author's opinion that Laakso intended to expand on his work in two manuscripts he was working on. In \cite{laakso2000ahlfors} he cited those manuscripts as forthcoming. Unfortunately, to the author's knowledge, T.J. Laakso retired from the Academic career and these manuscripts were never published, hence this remains only a speculation.
	
    With this in mind, the goal of this paper is to expand Laakso's ideas with some definitions and examples. We will also provide detailed proofs of \cite[Propositions 1.1 and 1.2]{laakso2000ahlfors}, that were missing in the original paper.
	
    The structure of this paper is the following: in section 2 we repeat step-by-step the construction of a Laakso space, with some in depth explanation on the reason behind some definitions. In section 3 we prove the main results from the first chapter of \cite{laakso2000ahlfors} about the structure of geodesics in a Laakso space. In section 4 a simple example of a Laakso space constructed from the classical middle-third Cantor set is provided. In the last section a list of possible further developments is also briefly discussed.
	
\section{Construction of a Laakso Space}
	
    The main idea behind Laakso's construction is to define a metric space as a quotient. The starting set is a Cartesian product $\C{F}\times I$, where $\C{F}$ is a Cantor-like fractal and $I=[0,1]$. From there, thanks to a set of identifications, Laakso obtain a path connected set endowed with a metric. A similar approach was also used, in a different context, also in \cite{barlow2004markov}.
	
    We start by recalling some notations and results from \cite{hutchinson1981fractals} that will be useful when dealing with coordinates on the fractal $\C{F}$.
	
    \begin{definition}
        Let $(X,d)$ be a complete metric space. An \emph{Iterated Function System} (IFS for short) in $(X,d)$ is a finite family of functions $\mathcal{F} = \{ f_0,\dots,f_n \}$ such that $r_i=\lip(f_i)\in(0,1)$ for each $i=0,\dots,n$. Functions of this kind are also called \emph{contractions}.
    \end{definition}
	
    Denote with $\B{K}(X)$ the set of all compact subsets of $X$. From an IFS $\mathcal{F}$ we can define a function $\cup \mathcal{F} : \mathbb{K}(X)\rightarrow \mathbb{K}(X)$ as:
    \[
        \cup \mathcal{F}(A) := \bigcup_{i=0}^n f_i(A).
    \]
    Since $\B{K}(X)$ equipped with the Hausdroff metric is complete whenever $X$ is complete, we can use Banach's fixed point theorem to define the attractor of an IFS.
	
    \begin{definition}\label{def_IFS}
        Let $\C{F}$ be an IFS in a complete metric space $(X,d)$. The \emph{attractor} of $\C{F}$ is the compact $A\subset X$ such that $\cup \C{F}(A) = A$. The attractor of an IFS is unique and is denoted by $|\C{F}|$ or, with a standard abuse of notation, simply by $\C{F}$.
    \end{definition}
	
    The main advantage of this definition is that it is easy to compute the Hausdorff dimension of the set constructed this way. As proved by Hutchinson in \cite{hutchinson1981fractals}, if we denote with $D$ the unique positive number for which $\sum_{i=0}^n r_i^D=1$, then the Hausdorff dimension of the attractor $\C{F}$ is exactly $D$.
	
    Indeed it is also possible to work the other way around: fix the desired $0<Q<1$ and define a set of Hausdorff dimension exactly $Q$ as the attractor of a suitable IFS. 
	
    Thanks to this approach we are able to put \textquotedblleft coordinates\textquotedblright\ on a fractal defined from an IFS. Let $\{f_0,\dots,f_n\}$ be the IFS and $\C{F}$ his attractor. With $\C{F}_i$ we denote the set $f_i(\C{F})$ for $i\in\{0,\dots,n\}$. In a similar way $\C{F}_{ij} = (f_j\circ f_i)(\C{F}) = f_j(f_i(\C{F}))$ for $i,j\in\{0,\dots,n\}$ and, more in general, if we have a string $a=a_1 a_2\dots$ with $a_i\in\{0,\dots,n\}$ then
    \begin{enumerate}[(1)]
	\item 
            if $a=a_1 \dots a_k$, $\C{F}_a = (f_{a_k}\circ\dots\circ    f_{a_1})(\C{F}) = f_{a_k}(\cdots f_{a_2}(f_{a_1}(\C{F}))\cdots)$ is a compact subset of $\C{F}$,
	\item
            if $a$ is of infinite length, $\C{F}_a$ is a point of $\C{F}$. In particular it is the unique fixed point of the composition $f_a:=\lim_{n\to\infty} f_{a_n}\circ\dots\circ f_{a_1}$.
    \end{enumerate}
	
    The interested readers can find more details in \cite{hutchinson1981fractals}.
	 
    \bigskip

    We can now proceed to show the construction of a general Laakso space.
	
    Choose $1<Q<2$ and $s$ satisfying $\log_{s}2 = Q - 1$. From this choice we have that $s>2$. Denote with $\C{F}$ the attractor of the IFS in $\B{R}$ defined by the functions
    \begin{equation}\label{eq_IFS}
        f_0(x) = \frac{x}{s}\ \  \text{ and }\ \  f_1(x) = \frac{x}{s} + \frac{s-1}{s}.
    \end{equation}
    Notice that it is a Cantor-like fractal with Hausdorff dimension $Q-1$. We now have the starting set in which we want to define the identifications: $\C{F}\times I$ with the metric induced by $\B{R}^2$.

    \begin{remark}\label{convention_vertial}
        In \cite{capolli2022maximal, steinhurst2010diffusions}, as well as in the original paper by Laakso, the Cartesian product for the starting set is $I\times \C{F}$. However, in those papers, the authors call the $I$-coordinate \emph{vertical}, causing some confusion. In this work, in order to be consistent with the fact that the vertical component is usually the second one, we will consider the Cartesian product for the starting set to be $\C{F}\times I$.
    \end{remark}

    To proceed with the construction of the identifications we choose the unique integer $n$ such that $n\leq s<n+1$ and we fix a sequence $\textbf{m} = \{m_i\}_{i\in\B{N}}$ such that the following conditions hold for every $i$:
    \begin{enumerate}[(1)]
        \item $m_i\in\{n,n+1\}$,
        \item $$\frac{n}{n+1}\prod_{j=1}^i m_j^{-1} \leq \frac{1}{s^i} \leq \frac{n+1}{n}\prod_{j=1}^i m_j^{-1}.$$
    \end{enumerate}
    The fact that at least one sequence with these properties exists follows from the condition $n\leq s<n+1$ and an algebraic computation.
    Notice that, in principle, such a sequence is not unique. Once $\textbf{m}$ has been fixed, we use it to define a family of functions $\{\omega_k\}_{k\in\B{N}}$ as following:
    \begin{enumerate}[(1)]
	\item $\omega_1:\{1,\dots,m_1-1\}\to(0,1)$,
	\item
            for $k\geq2$, $\omega_k:\{0,\dots,m_1-1\}\times\dots\times\{0,\dots,m_{k-1}-1\}\times\{1,\dots,m_k-1\}\to(0,1)$,
	\item for every $k\in\B{N}$,
            \begin{equation}\label{omega}
		      \omega_k(n_1,\dots,n_k) = \sum_{j=1}^k n_j\prod_{h=1}^j m_h^{-1}.
		\end{equation}
    \end{enumerate}
    It can be easily checked that the functions just defined are injective for each $k\in\B{N}$.
	
    \begin{definition}
        We call $\omega\in [0,1]$ a \emph{wormhole level of order} $k$ if $\omega = \omega_k(n_1,\dots,n_k)$ for $k\in\B{N}$ and for certain $$(n_1,\dots, n_k)\in \{0,\dots,m_1-1\}\times\dots\times\{0,\dots,m_{k-1}-1\}\times\{1,\dots,m_k-1\}.$$
		
        For $k\in\B{N}$, the set of all wormhole levels of order $k$ will be denoted with $J_k$.
    \end{definition}
    \begin{remark}
        The condition that the last entry of $\omega_k$ cannot be 0 is crucial in order to avoid the overlapping of wormhole levels. This means that, if $k\neq h$, then the image of $\omega_k$ is disjoint from the image of $\omega_h$.
    \end{remark}

    The natural number $k$, which is uniquely determined, will often be called the \emph{order} of the wormhole level. Sometimes we will omit the $k$ in $\omega_k(n_1,\dots,n_k)$ as it will be clear from the number of variables. 
	
    From the way the parameters are chosen, follows a nesting property that will be crucial in the following. 
	
    \begin{lemma}\label{nested_wormhole}
        Let $\omega$ be a wormhole of level $N$. Then for every $M>N$ there exists $\omega^M,\theta^M\in J_M$ such that $0<\theta^M<\omega<\omega^M<1$.
		
        Similarly, let $N_1\leq N_2$ be natural numbers and consider two distinct wormhole levels $\omega^1\in J_{N_1}$ and $\omega^2\in J_{N_2}$. Then, for every $M>N_2$, there exists $\omega^M\in J_M$ such that either $\omega^1 < \omega^M < \omega^2$ or $\omega^2 < \omega^M < \omega^1$, depending of whether $\omega^1<\omega^2$ or vice versa.
    \end{lemma}
	
    \begin{proof}
        Let $\omega = \omega(n_1,\dots,n_N).$ Then, for $M>N$, we define the wormhole level $\omega^M := \omega_M(n_1,\dots,n_N,0,\dots,0,1)$. It is the wormhole level defined by $M$ entries, the first $N$ of which are equal to the entries in $\omega$, then all the others are 0s except for the $M$-th, which is 1. If follows from how the function are defined in \eqref{omega} that $\omega<\omega^M<1$. Similarly we define $\theta^M:=\omega_M(n_1,\dots,n_N-1,0,\dots,0,1)$. It is defined exatly as $\omega^M$ except in the $N$-th entry, which is reduced by one. This is possible because $n_N\neq 0$ by definition. With a simple computation we see that $0<\theta^M<\omega$, hence proving the first assertion.
		
        To prove the second assertion take $\omega^1=\omega(n_1,\dots,n_{N_1})$, $\omega^2=\omega(n'_1,\dots,n'_{N_2})$ and suppose that $\omega^1<\omega^2$, the proof in the other case is identical. Then, for $M>N_2$, we define the wormhole level $\omega^M$ in a similar way as in the first part of the proof: $\omega^M:=\omega_M(n'_1,\dots, n'_{N_2}-1,0,\dots, 0, 1)$. From an easy computation follows immediately that $\omega^1 < \omega^M < \omega^2$ as requested.
    \end{proof}
	
    We are now ready to define the identifications. The idea is that the Cantor-like fractal $\C{F}$ can be seen as formed by cells of various order. For example $\C{F}_0$ and $\C{F}_1$ are the two cells of order 1 and $\C{F}_{00},\C{F}_{01},\C{F}_{10}$ and $\C{F}_{11}$ are the four cells of order 2. More in general there are $2^n$ cells of order $n$ and they are denoted by $\C{F}_a$, where $a$ is a string of length $n$. Wormholes levels of order $k$ will be used to jump among successive cells of the same order. For example, wormhole levels of order 1 are used to pass from $\C{F}_0\times I$ to $\C{F}_1\times I$. To move inside smaller cells, e.g. from $\C{F}_{00}\times I$ to $\C{F}_{01}\times I$, we will need to use wormhole level of increasing order. Notice that wormholes of order 2 can be used to move from $\C{F}_{00}\times I$ to $\C{F}_{01}\times I$ or from $\C{F}_{10}\times I$ to $\C{F}_{11}\times I$ but no to move, for example, from $\C{F}_{00}\times I$ to $\C{F}_{11}\times I$. More in general a wormhole level of order $k$ can be used to move from $\C{F}_{a0}\times I$ to $\C{F}_{a1}\times I$, where $a$ is a string of length $k-1$.
	
    Wormhole levels are used to define an equivalence relation on $\C{F}\times I$.
	
    \begin{definition}\label{defident}
        Let $(x_1,y_1), (x_2,y_2)\in\C{F}\times I$. We say that $(x_1,y_1)\sim (x_2,y_2)$ if and only if the following conditions hold:
        \begin{enumerate}[(1)]
            \item $x_1 = x_2 \pm \frac{s-1}{s^k}$, where the sign depends on whether $x_1>x_2$ or $x_1<x_2$,
            \item $y_1 = y_2 = \omega_k(n_1,\dots,n_k)$ is a wormhole level of order $k$.
	\end{enumerate}
	
        We call $\pi$ the identification map defined as $\pi(x_1,y_1):=[x_1,y_1]$, where square bracket denotes equivalence classes in $\C{F}$. 
        The \emph{Laakso space} associated to $\C{F}$ and $\pi$ is then $\C{L}:=\pi(\C{F}\times I)$. Points in $\C{L}$ will be denoted as $[x_1,y_1]$.	The topology on $\C{L}$ is the quotient topology inherited from the Euclidean topology on $\C{F}\times I$.
    \end{definition}
	
    As proved by Laakso in \cite{laakso2000ahlfors}, $\C{L}$ is compact and of Hausdorff dimension $Q$. 
	
    $\C{L}$ is equipped with a natural projection on the vertical coordinate: $h:\C{L}\to [0,1]$ defined as $h([x_1,y_1]):= y_1$. This is well defined because points in $\C{F}\times I$ that are mapped in the same equivalence class by $\pi$ have the same $I$-coordinate. We call $h$ the height function. This name was already used by other authors but it makes more sense in light of Remark \ref{convention_vertial}. 
 
    The same notation will be used also for the projection on the second coordinate on $\C{F}\times I$. Weather we are considering $h$ as a function from $\C{L}$ or from $\C{F}\times I$ will be clear from the context.
	
    \begin{remark}
        Hidden in condition $(1)$ there is the fact that only points belonging to consecutive cells of the same order are identified. This means that, if two points $(x_1,y_1),(x_2,y_2)\in \C{F}\times I$ are projected to the same point for a certain wormhole level of order $k$, then $x_1\in \C{F}_{a0}$ and $x_2\in \C{F}_{a1}$ for some string $a$ of length $k-1$.
    \end{remark}

    There is an easy way for visualizing the action of the identification map $\pi$ with the help of the coordinate system on $\C{F}$ explained after Definition \ref{def_IFS}. We first introduce the $n$-\emph{th switching function} $\nu^n$. It acts on strings of length at least $n$ by taking $a=a_1a_2\dots$ and transforming it into the new string $\nu^n(a)=a'_1 a'_2\dots$ where
    \[
        a'_m = 	\begin{cases} 
				a_m \ \ if\ m\neq n\\
				1 \ \ \ \ if\ m=n\ and\ a_n=0\\
				0\ \ \ \ if\ m=n\ and\ a_n=1.
			\end{cases}
    \]
	
    Now take a point $x_1=\C{F}_a\in\C{F}$ for an infinite string $a=a_1 a_2 \dots$. A wormhole level $y_1=\omega(n_1,\dots,n_k)$ of order $k$ is used to identify the point $(x_1,y_1)\in\C{F}\times I$ with the point $(x_2,y_1)\in\C{F}\times I$, where $x_2=\C{F}_{b}\in\C{F}$ with $b=\nu^k(a)$. Hence, by moving each time to the appropriate height (possibly infinitely many times), we can change any $x_1=\C{F}_a$ into any other $x_2=\C{F}_b$. This combination of moving in the $I$ direction and jumping with wormhole levels is at the base of how paths in $\C{L}$ are defined.

\section{The metric on $\C{L}$}

    This section is devoted to the study of the metric of $\C{L}$. For the \textquotedblleft measure\textquotedblright\ part of the study we refer to \cite[Section 5]{capolli2022maximal}. In \cite{laakso2000ahlfors} Laakso defined a distance on $\C{L}$ and described the geodesics starting from the concept of path. As expected, a path that connects points $x,y\in\C{L}$ is (the image of) a continuous function $p:[0,1]\to\C{L}$ such that $p(0)=x$ and $p(1)=y$. The intuition suggests that, in the Laakso space, such path can be seen as a collection of vertical segments connected via wormholes. 
	
    A distance can then be defined as follows.
    
    \begin{definition}
	Let $x,y\in \C{L}$. Then the distance between $x$ and $y$ is
		\begin{equation}\label{def_dist}
		  |x-y|:=\inf\{\C{H}^1(\Gamma)\,|\,\pi(\Gamma)\text{ is a path joining }x\text{ and }y \}.
		\end{equation}
    \end{definition}

    For the rest of the paper we will denote the distance in $\C{L}$ with $d(x,y)$ instead of $|x-y|$. This is to avoid confusion with the distance in the vertical component (wich is the usual Euclidean distance), that will also play a fundamental role.
	
    The main goal of this section is to give a rigorous proof of \cite[Proposition 1.1]{laakso2000ahlfors} and \cite[Proposition 1.2]{laakso2000ahlfors}. Those results, despite being key when working with paths and geodesics in the Laakso space, were not proven in the original paper by Laakso.

    From now on, when needed, we will identify a point $x\in\C{F}$ with the unique infinite string $a=a_1 a_2 \dots$ such that $x=\C{F}_a$.
    \begin{definition}\label{defasympt}
        We say that $x_1,x_2\in\C{F}$ have \emph{the same asymptotic behaviour} if there exists a positive integer $n$ such that $(x_1)_i=(x_2)_i\ \forall i\geq n$, i.e. if $x_1$ and $x_2$, viewed as strings, are eventually equal. We will denote this fact with $x_1\parallel x_2$.
	
        If such an $n$ does not exists, then we say that $x_1$ and $x_2$ have \emph{different asymptotic behaviour} and we will denote this fact with $x_1\nparallel x_2$.
    \end{definition}

    The construction of a path connecting two points in a Laakso space depends on the asymptotic behaviour of the $\C{F}$-coordinates of the two points we are considering.
	
    \begin{remark}
        When we speak of the $\C{F}$-coordinate (or the $I$-coordinate) of a point $x\in\C{L}$ we are actually speaking of the corresponding coordinate of $\pi^{-1}(x)$ in $\C{F}\times I$. This is a standard abuse of notation.
    \end{remark}

    \begin{proposition}\label{proppath}
        Let $x=[x_1,y_1]$ and $y=[x_2,y_2]$ be two distinct points in $\C{L}$. Then there exists a path $p:[0,1]\to\C{L}$, connecting the two points, that is the image under the map $\pi$ of a family of countably many\footnote{Here by countably many we mean finite or countable infinite.} (closed) vertical line segments $\Gamma\subset\C{F}\times I$.
    \end{proposition}

    \begin{proof}
        The proof is a constructive algorithm and it gives us a way to find both the path and the set of vertical line segments in $\C{F}\times I$ from which it comes from.

        From now on, since there is no risk of confusion, we will just say line segments when speaking of vertical line segments in $\C{F}\times I$.
	
		Take two points as in the hypothesis and assume that neither $y_1$ or $y_2$ is a wormhole level. 

        If $x_1=x_2$ then the two points are one on the vertical of the other. A path to join them is then simply $\pi(\gamma(t))$, where $\gamma(t):[0,1]\to \C{F}\times I$ and $\gamma([0,1])$ is a line segments that connects $\pi^{-1}(x)$ to $\pi^{-1}(y)$\footnote{Those are well defined points in $\C{F}\times I$ since we already assumed that $y_1$ and $y_2$ are not wormhole levels}. 

        If the two points are not one on the vertical of the other, i.e. if $x_1\neq x_2$, then we can proceed with the following algorithm:
	\begin{enumerate}[(1)]
		\item
                Call $q_0=(x_1,y_1)\in\C{F}\times I$ and define $i_1$ to be the smallest $i\in\B{N}$ such that $(x_1)_i\neq (x_2)_i$. Notice that $x_2$ will remain fixed through the algorithm, while $x_1$ will change after every iteration.
		\item
                Choose the wormhole level of order $i_1$ closest to $y_1$ (it can be either above or below height $y_1$). Let us say it is $\omega^1= \omega_{i_1}(n_1,\dots,n_{i_1})$.
		\item
                Define a line segment in $\C{F}\times I$ that starts from $q_0$ and ends in $(x_1,\omega^1)$. One way of doing this is as the image of the linear function $\gamma_1:[0,1]\to\C{F}\times I$ defined as $\gamma_1(t) = (x_1, y_1+t(\omega^1-y_1))$. Since $\omega^1$ is a wormhole level of order $i_1$, the point $(x_1,\omega^1)$ will be identified by the map $\pi$ with the point $q_1:=(x_1',\omega^1)$, where $x_1'=\nu^{i_1}(x_1)$.
		\item 
                If $x_1'\neq x_2$, then restart the algorithm by replacing $q_0$ with $q_1$ and $x_1$ with $x_1'$ in step (1), finding $\omega^2$ in step (2) and defining $\gamma_2(t)$ and $q_2$ in step (3). More in general, at the $j$-th iteration, we look for a wormhole of level $i_j$, a linear function $\gamma_j$ that corresponds to a line segment connecting $(x_1^{(j-1)},\omega^{j-1})$ and $(x_1^{(j-1)},\omega^j)$ and a point $q_j:=(x_1^{(j)},\omega^j)$.
        \end{enumerate}
        At this point we have two possible scenario:
        \begin{itemize}
            \item
                if $x_1\parallel x_2$, then the algorithm will stop after a finite number of iterations. 
            \item 
                if $x_1\nparallel x_2$, then the algorithm will go on indefinitely.
        \end{itemize}

        \noindent \emph{Case 1):} Let us assume that the algorithm stops after $k-1$ iterations, i.e. $x_1^{(k-1)}=x_2$. At this point we are on the vertical of $(x_2,y_2)$ (notice that we cannot have reached it, because all the points $(x_1^{(j)},\omega^j)$ are wormholes, while $(x_2,y_2)$ is not). We then need a last line segment, connecting $q_{k-1}=(x_2,\omega^{k-1})$ and $(x_2,y_2)$. It will correspond to the linear function $\gamma_k(t)$.

        Now that we have all the linear functions $\gamma_i$ for $i=1,\dots k$ we define
		\begin{equation*}
			\Gamma(t):=  \begin{cases}
                                \gamma_1(kt)& \text{if } t\in[0,\frac{1}{k}]\\
                                \vdots & \\
                                \gamma_j(kt-j+1)& \text{if } t\in[\frac{j-1}{k},\frac{j}{k}]\\
                                \vdots & \\
                                \gamma_k(kt-k+1)& \text{if } t\in[\frac{k-1}{k},1].
	                       \end{cases}
		\end{equation*}
        Notice that $\Gamma:[0,1]\to\C{F}\times I$ is such that $\Gamma(0)=x$ and $\Gamma(1)=y$. Moreover $\Gamma([0,1])\subset\C{F}\times I$ is a finite union of (closed) line segments, hence we are missing only the continuity. 
	\begin{claim}
            $p(t):=\pi(\Gamma(t))$ is continuous.
	\end{claim}
	
	\emph{Proof of Claim:}
            By construction each $\gamma_j$ is continuous, hence, since $\pi$ does nothing to the vertical coordinate, we have to check $p(t)$ only for the values of $t$ in which the paths are joined, namely $t=\frac{j}{k}$ for $i=1,\dots,k-1$. From how the $\gamma_j$ were constructed we get that, for each $j=1,\dots,k-1$:
		\[
                \lim_{t\to \frac{j}{k}^+}p(t)=\lim_{t\to \frac{j}{k}^+}\pi(\Gamma(t)) = \lim_{t\to \frac{j}{k}^+}\pi(\gamma_{j+1}(kt-j))=\pi(\gamma_{j+1}(0))=\pi((x_1^{(j-1)},\omega^j))
		\]
		and also
		\[
                \lim_{t\to \frac{kj}{k}^-}p(t)=\lim_{t\to \frac{j}{k}^-}\pi(\Gamma(t)) = \lim_{t\to 	\frac{j}{k}^-}\pi(\gamma_j(kt-j+1))=\pi(\gamma_j(1))=\pi((x_1^{(j)},\omega^j)).
		\]
            The continuity follows from the fact that, since $\omega^j$ is a wormhole of level $j$ and $\nu^{i_j}(x_1^{(j-1)})=x_1^{(j)}$, then $\pi((x_1^{(j-1)},\omega^j)) = \pi((x_1^{(j)},\omega^j))$. This proves the claim and the first case. \hfill
        \qedsymbol\medskip

        \noindent \emph{Case 2):} The intuition in the case $x_1\nparallel x_2$ suggests that our path will be the image under the map $\pi$ of a countable number of line segments.

        The algorithm proceeds exactly as before, with the exception that it never stops. What happens instead is that it produces an infinite countable number of line segments in $\C{F}\times I$ whose end points form a sequence $\{ (x_1^{(j)},\omega^j) \}_{j\in\C{J}}\subset\C{F}\times I$ with $\C{J}=\{j_h\}_{h\in\B{N}}\subset\B{N}$.
		
	\begin{claim}\label{claim_infinite}
            The end points of the line segments $\gamma_j$ are converging to $(x_2,\overline{\omega})\in\C{F}\times I$.
	\end{claim}
	
	\emph{Proof of Claim:}
            The fact that $x_1^{(j)}\to x_2$ as $j\to\infty$ is clear from how the $x_1^{(j)}$ are defined, i.e. by changing the string $x_1$, one entry at a time, to make it coincide with $x_2$. 
		
            To show the convergence of the second coordinate we recall how wormhole levels were defined, and in particular \eqref{omega}. At the $k$-th iteration of the algorithm we are adding the term $n_{j_k} \prod_{h=1}^k m_h^{-1}$, with $n_{j_k}\in\{1,\dots,m_k-1\}$. Hence, since $m_h\geq 2$ for each $h$, we can apply the Cauchy criterion to see that the sequence $(\omega^j)_j$ is converging to some $\overline{\omega}\in [0,1]$.
        \hfill \qedsymbol
		
        We now define a line segment with extremes $(x_2,\overline{\omega})$ and $(x_2,y_2)$ that correspond to a function $\gamma_{\infty}:[0,1]\to\C{F}\times I$. Once we have all the $\gamma_j$ we define
	\begin{equation*}
		\Gamma(t):=
		\begin{cases}
			\gamma_1(4t)& \text{if } t\in[0,\nicefrac{1}{4}]\\
			\gamma_2(8t-2)& \text{if } t\in[\nicefrac{1}{4},\nicefrac{1}{4}+\nicefrac{1}{8}]\\
			\vdots & \\
                \gamma_j(2^{j+1}t-\sum_{s=1}^{j-1}2^s)& \text{if } 	t\in\left[\sum_{s=1}^{j-1}\nicefrac{1}{2^{s+1}},\sum_{s=1}^j\nicefrac{1}{2^{s+1}}\right]\\
			\vdots & \\
			\gamma_{\infty}(2t-1)& \text{if } t\in[\nicefrac{1}{2},1].
		\end{cases}
	\end{equation*}
        The re-parametrization used here is taking into account the fact that the later we encounter a line segments, the shorter that line segment will be. This is because wormhole levels of higher order are closer to each other. Clearly this is only one of the many possible ways to parametrize $\Gamma(t)$.
		
        Notice that this time $\Gamma([0,1])$ is a countable collection of line segments in $\C{F}\times I$.
		
	\begin{claim}
            $p(t):=\pi(\Gamma(t))$ is a continuous path in $\C{L}$ that joins $x$ and $y$.
	\end{claim}
		
	\emph{Proof of Claim:}
            Clearly $p(0) =\pi(\gamma_1(0))=x$ and $p(1)=\pi(\gamma_{\infty}(1))=y$, hence we are left to check the continuity. The proof is the same as in the finite case except for the point $t=\frac{1}{2}$, that correspond to the special path $\gamma_{\infty}(t)$. To prove continuity at that point we first observe that
		\[
                \lim_{t\to \frac{1}{2}^+}p(t)=\lim_{t\to \frac{1}{2}^+}\pi(\Gamma(t)) = \lim_{t\to \frac{1}{2}^+}\pi(\gamma_{\infty}(2t-1))=\pi(\gamma_{\infty}(0))=\pi((x_2,\overline{\omega}))=[x_2,\overline{\omega}].
		\]
		For the other direction we get
		\begin{align*}
                \lim_{t\to \frac{1}{2}^-}p(t)&=\lim_{t\to \frac{1}{2}^-}\pi(\Gamma(t)) =\lim_{j\to\infty}\left( \lim_{t\to T(j)} \pi\left( \gamma_j(2^{j+1}t-\sum_{s=1}^{j-1}2^s) \right) \right)\\
                & =\lim_{j\to\infty}(\pi(\gamma_j(1))) =\lim_{j\to\infty}[x_1^{(j)},\omega^j] = [x_2,\overline{\omega}]			
		\end{align*}
            where $T(j) = \sum_{s=1}^j\nicefrac{1}{2^{s+1}}$ comes from the definition of $\Gamma(t)$. This concludes the proof of the claim. \hfill \qedsymbol
			
        To finish the proof of the proposition we are left to see what happen when $y_1$ and/or $y_2$ are wormhole level. Suppose $y_1$ is a wormhole level and assume in particular that it is of order $m$, the other cases are similar. Then we can find $x_1$ and $\hat{x}_1=\nu^m(x_1)$ such that $\pi^{-1}([x_1,y_1])=\{(x_1,y_1),(\hat{x}_1,y_1)\}$. We then choose as $q_0$ one of the two and we do it in such a way that $(x_1)_m$ (or $(\hat{x}_1)_m$) is equal to $(x_2)_m$. With this choice the algorithm then proceeds as in the cases presented above.
    \end{proof}
	
    \begin{remark}\label{remark_order}
        There is no need to take the wormhole in the order dictated by the proof. The order used in the algorithm makes sure that no wormhole is missed, however in doing so there is no guarantee that the resulting path is the shortest (see the next section for an example in this sense).
    \end{remark}

    \begin{remark}
        We can define the functions $\gamma_j(t)$ in various different ways. However, in order to keep the computations simple, we will always take them to be injective.
    \end{remark}
 
    As a natural consequence of Propositions \ref{proppath} we get the following
	
    \begin{corollary}
        $\C{L}$ is a path-connected metric space.
    \end{corollary}

    Later we will prove the stronger fact that $\C{L}$ is also a geodesic metric space, meaning that any two points can be connected by a path of minimal length.

\subsection{Geodesics}\ \\

    Any path $p\subset\C{L}$ comes naturally with a lift, i.e. a set $\Gamma\subset\C{F}\times I$ such that $\pi(\Gamma)=p$. Such a lift is, in principle, not unique. However, once a lift has been specified, we can use the notation $(p,\Gamma)$ to refer to the path and it becomes possible to define its length.

    \begin{definition}\label{def:length}
        The length of a path $(p,\Gamma)$ is defined as
        $$l(p,\Gamma):=\C{H}^1(\Gamma).$$
        Moreover, if $(p,\Gamma)$ is constructed as in Proposition \ref{proppath}, we also have that $$\C{H}^1(\Gamma)=\sum |h(\gamma_j(1))-h(\gamma_j(0))|$$ where the sum ranges over all the $\gamma_j$ that are used to define $\Gamma$.
    \end{definition}

    \begin{remark}
        Notice that the summation will always converge. Indeed the term $|h(\gamma_j(1))-h(\gamma_j(0))|$ is comparable to $\frac{1}{s^j}$, i.e. the summation is comparable to the series $\sum_{i\in\B{N}} \frac{1}{s^i}$, which converges since $s>2$.
    \end{remark}
    
    As already observed by Laakso in his paper, there could exist paths whose lift is a totally disconnected set in $\C{F}\times I$. In his paper he said that it is possible to ignore these kind of path, but he did not explain why or how. In the following we are going to show that, in order to compute the distance among two points connected by a monotone path, it suffice to consider paths constructed as in Proposition \ref{proppath} (upon rearrangement, as explained in Remark \ref{remark_order}). Later in the paper we will also show that this is true for every pair of points in $\C{L}$, dropping the monotonicity assumption.

    Recall that, from Proposition \ref{proppath}, $\Gamma([0,1])=\{\gamma_j([a_j,b_j])\}_j\subset\C{F}\times I$ is a family of line segments, where the intervals $[a_j,b_j]$ comes from on of the re-parametrizations in the proof. From now on, in order to ease the notation and when there is no risk of confusion, we will simply write $\Gamma$ or $\Gamma=\{\gamma_j\}_j$ to indicate the family of line segments $\Gamma([0,1])$. 

    We now clarify some intuitions from \cite{laakso2000ahlfors} about the behaviour of paths.
	
    For a path $(p,\Gamma))$ where $\Gamma=\gamma([0,1])\subset\C{F}\times I$ is just a segment, the meaning of upward and downward is clear: if $h(\pi(\gamma(t_1)))<h(\pi(\gamma(t_2)))$ whenever $t_1<t_2$ then $(p,\Gamma)$ \emph{goes upward} and vice versa for the definition of downward. The delicate part is when the path in $\C{L}$ is the image of multiple line segment in $\C{F}\times I$ connected via wormholes.

    \begin{definition}\label{defpassaggio}
        Take a Laakso space of dimension $Q=1+\log_s 2$ and let $(p,\Gamma)$ be a path in $\C{L}$. We say that $(p,\Gamma)$ \emph{passes (or jumps) through a wormhole level of depth} $k$ if there exists $q\in\C{L}$ such that $\pi^{-1}(q)=\{(x_1,y_1),(x_2,y_2)\}\subset\C{F}\times I$ and:
	\begin{enumerate}[(1)]
            \item $|x_1-x_2| =  s^{1-k}\frac{s-1}{s}$,
		\item $y_1 = y_2 = \omega_k(n_1,\dots,n_k)$ is a wormhole level of depth $k$,
		\item $\exists j$ such that $\gamma_j(1) = (x_1,y_1)$ and $\gamma_{j+1}(0) = (x_2,y_2)$.
	\end{enumerate}
	Moreover we say that:
	\begin{enumerate}[(1)]\setcounter{enumi}{3}
            \item $(p,\Gamma)$ passes through $q$ \emph{going upward} if $h(\gamma_j(0))<y_1<h(\gamma_{j+1}(1))$,
            \item $(p,\Gamma)$ passes through $q$ \emph{going downward} if $h(\gamma_j(0))>y_1>h(\gamma_{j+1}(1))$,
		\item $(p,\Gamma)$ \emph{makes an inversion} at $q$ in the other cases.
	\end{enumerate}
        The path $(p,\Gamma)$ \emph{goes only upward} (or only downward) if, for each wormhole $q$ in which it passes, it does so by going upward (or downward). Such paths are also called \emph{monotone}. Finally, if $(p,\Gamma)$ is not monotone, then we call it \emph{oscillating}.
    \end{definition}

    Before discussing the main result of this section, we now give another formal definition of an idea already present in \cite{laakso2000ahlfors}.
	
    \begin{definition}\label{def_minimalinterval}
	Let $x,y\in\C{L}$. An interval $[a,b]\subseteq[0,1]$ such that:
	\begin{enumerate}[(1)]
		\item $h(x),h(y)\in[a,b]$,
            \item $[a,b]\cap J_N\neq\emptyset$ for each $N\in\B{N}$ such that a wormhole of level $N$ is required to connect $x$ to $y$,
            \item if $[a',b']\subseteq[0,1]$ is another interval that satisfies properties (1) and (2), then $b-a\leq b'-a'$,
	\end{enumerate}
	is called \emph{minimal height interval} (or simply \emph{minimal interval}) for $x$ and $y$.
    \end{definition}
	
    Monotone paths and minimal intervals are related by the following
	
    \begin{lemma}\label{lemma:monotoneminimal}
        Two points $x,y\in\C{L}$ can be connected with a monotone path if and only if their minimal interval is $[h(x),h(y)]$ (or $[h(y),h(x)]$ if $h(x)>h(y)$).
    \end{lemma}
	
    \begin{proof}
        Suppose that $x=[x_1,y_1]$ and $y=[x_2,y_2]$, with $h(x)<h(y)$, are connected by a monotone path. Then for any wormhole level needed to connect them there is at least one wormhole of that level between heights $h(x)$ and $h(y)$. This means that the interval $[h(x),h(y)]$ satisfies (1) and (2) of Definition \ref{def_minimalinterval}. Since any other interval that satisfies (1) and (2) must contain $h(x)$ and $h(y)$, we have that (3) is also satisfied, which implies that $[h(x),h(y)]$ is a minimal interval as requested.
		
        On the other hand let us suppose that the minimal interval for $x,y\in\C{L}$ is $[h(x),h(y)]$. Let $j_1$ and $j_2$ with $j_1<j_2$ be the two smallest order of wormhole levels needed to connect $x$ to $y$. Since $[h(x),h(y)]$ is a minimal interval, then there exist both $\omega^1\in [h(x),h(y)]\cap J_{j_1}$ and $\omega^2\in[h(x),h(y)]\cap J_{j_2}$. Let us further assume that $\omega^1<\omega^2$, the other case is identical. Then we can construct a path that connects $x$ and $y$ in the following way:
	\begin{enumerate}[(1)]
		\item We connect $x$ to $x'=[x_1,\omega^1]$ with a single, upward going line segment.
            \item We use Proposition \ref{proppath} to connect $x'$ to $y'=[x_2,\omega^2]$ with a monotone path. This is possible because, by repeatedly applying Lemma \ref{nested_wormhole}, we can take the jumps in increasing order.
		\item We connect $y'$ to $y$ with a single, upward going, line segment.
	\end{enumerate}
        The concatenation of these paths is a monotone, upward going path that connects $x$ to $y$ as required.
	\end{proof} 
 
    We are now able to prove a formula for the length of a monotone path. Then we will prove \cite[Proposition 1.1]{laakso2000ahlfors}, showing the general form of a geodesic in $\C{L}$. Finally, we will show how some of the intermediate results will combine to prove \cite[Proposition 1.2]{laakso2000ahlfors}. 
	
    \begin{proposition}\label{prop_distancemonotone}
        Let $x=[x_1,y_1]$ and $y=[x_2,y_2]$ be two distinct points in $\C{L}$ that can be connected by a monotone path. Then $d(x,y)=|h(y)-h(x)|$.
    \end{proposition}

    \begin{proof}
        We prove the proposition in the case $h(x)<h(y)$, the other case is identical. 
        
        First we observe that, for any path $(p,\Gamma)$ that connects $x$ and $y$, we must have that $[h(x),h(y)]\subseteq h(p)$. Hence
        \begin{equation*}
            |h(y)-h(x)|=\C{H}^1([h(x),h(y)])\leq\C{H}^1(h(p))
        \end{equation*}
        from which we infer that
        \begin{equation}\label{eq:inequalitydistance}
            |h(y)-h(x)|\leq d(x,y).
        \end{equation}
        Notice that, in order to prove this inequality, we did not used the hypothesis of the Proposition. Indeed \eqref{eq:inequalitydistance} holds for any two points $x,y\in\C{L}$, not only for those that can be connected by monotone paths.
        
        To prove the reverse inequality we observe that, thanks to Lemma \ref{lemma:monotoneminimal}, since we are able to connect $x$ and $y$ with a monotone interval, all the wormholes needed to connect them are between heights $h(x)$ and $h(y)$. Hence, by using the algorithm in Proposition \ref{proppath}, we can construct a monotone path $(p,\Gamma)$ with $\Gamma = \{\gamma_j\}_j$. For this path we have $$d(x,y) \leq l(p,\Gamma)= \sum|h(\gamma_j(1))-h(\gamma_j(0))| = |h(y)-h(x)|$$
        where the first equality comes from Definition \ref{def:length} and the second one from resolving the telescopic sum. This, together with \eqref{eq:inequalitydistance}, concludes the proof.
    \end{proof}


    \begin{remark}
        What the corollary is telling us is that, whenever we are looking at points that can be connected my monotone paths, then we can construct a geodesic by using the algorithm of Proposition \ref{proppath}. This means in particular that, as stated by Laakso in \cite{laakso2000ahlfors}, we can ignore those paths whose pre-image is a totally disconnected set in $\C{F}\times I$. At the end of this section we will be able to say the same thing for any two generic points in $\C{L}$.
    \end{remark}
	
    A crucial step to be able to prove \cite[Proposition 1.1]{laakso2000ahlfors} is to show that it is always possible to connect two points in $\C{L}$ by a path that makes at most two inversions.
	
    \begin{proposition}\label{prop_inversions}
        Let $x,y\in\C{L}$, Then there exists a path that connects $x$ and $y$ and makes at most two inversions. 
    \end{proposition}
	
    \begin{proof}
        Take $x=[x_1,y_1]$ and $y=[x_2,y_2]$ in $\C{L}$ and suppose that $h(x)<h(y)$. Assume that $x$ and $y$ cannot be joined by a monotone path and define
	\[
            \C{J}:=\{n\in\B{N}\,|\,\text{a wormhole of level } n \text{ is needed to connect } x \text{ to } y\}.
	\]
        Order $\C{J}$ as $\C{J}=\{n_1,n_2,\dots\}$ with $n_1<n_2<\dots$ and let $j$ be the minimal index such that there are no wormholes of order $n_j$ in $[h(x),h(y)]$. Notice that this $j$ must exist because we assumed that the two points cannot be joined by a monotone path. Let $\omega_j$ be the wormhole level of order $n_j$ closest to $[h(x),h(y)]$ (choose one if there are two at the same distance). We further assume that $\omega_j<h(x)$. The other case is $\omega_j>h(y)$ (recall that we just escluded the possibility $h(x)\leq\omega_j\leq h(y)$) and the proof in that case is identical since we are only interested in inversions. We further distinguish two cases.
		
        \emph{Case 1:}
            All the wormhole levels needed to connect $x$ and $y$ are in $[\omega_j,h(y)]$. We show how to construct a path that connect $x$ and $y$ and makes only one inversion. First we connect $x$ to the point $[x_1,\omega_j]$. This can be done with a monotone downward going path $(p',\Gamma')$ with $\Gamma'$ a single line segment in $\C{F}\times I$. From there we go to $y$ with a monotone upward going path $(p'',\Gamma'')$, which is possible thanks to how we choose $\omega_j$ and to Corollary \ref{lemma:monotoneminimal}. It follows immediately that the path obtained by concatenation\footnote{By concatenating two paths we mean that we are defining, via re-parametrizations, a new continuous $\Gamma:[0,1]\to\C{F}\times I$ such that $\Gamma([0,1]) = \Gamma'([0,1])\cup\Gamma''([0,1])$.} connects $x$ to $y$ and makes only one inversion at $[x_1,\omega_j]$.
		
	\emph{Case 2:}
            There are still some wormhole levels needed to connect $x$ and $y$ that cannot be found in $[\omega_j,h(y)]$. If this is the case then one of these wormhole levels must be $n_{j+1}$. Indeed, if a wormhole of level $n_{j+1}$ is in the interval $[\omega_j,h(y)]$, then, by Lemma \ref{nested_wormhole}, in the same interval we can find a wormhole of level $m$ for every other $m>n_{j+1}$, which would take us back to case 1. Take $\omega_{j+1}$ to be the first wormhole level of order $n_{j+1}$ above $h(y)$, which exists thanks to Corollary \ref{lemma:monotoneminimal}. We show how to construct a path from $x$ to $y$ that makes only two inversions. First we connect $x$ to $[x_1,\omega_j]$ with a monotone downward going path, as in case 1. Then we connect $[x_1,\omega_j]$ to the point $[x_2,\omega_{j+1}]$. Since the minimal interval for this two points is $[\omega_j,\omega_{j+1}]$, it can be done with a monotone upward going path. Finally we connect $[x_2,\omega_{j+1}]$ to $y$ with a monotone path, which must be downward going (recall that $\omega_{j+1}>h(y)$). It follows immediately that the path obtained by concatenation connects $x$ to $y$ and makes only two inversions, at $[x_1,\omega_j]$ and at $[x_2,\omega_{j+1}]$, hence proving the proposition. 
    \end{proof}
	
    Notice that the intervals $[\omega_j,h(y)]$ and $[\omega_j,\omega_{j+1}]$ are, in their respective case, minimal interval for the points $x$ and $y$. This follows immediately from how the wormholes were chosen during the construction. 
	
    \begin{definition}
        A path $(p,\Gamma)$ constructed starting from a minimal interval as in the proof of Proposition \ref{prop_inversions} is called \emph{path associated to the minimal interval}.
    \end{definition}
	
    \begin{corollary}\label{coroll_length}
        Let $x,y\in\C{L}$ with $h(x)\leq h(y)$ and let $[a,b]$ be a minimal interval for these points. The length of a path $(p,\Gamma)$ associated to the minimal interval is given by
	\[
		l(p,\Gamma) = 2b -2a - h(y) + h(x).
	\]
        On the other hand, if $h(x)>h(y)$, then the length of a path $(p,\Gamma)$ associated to the minimal interval is given by
	\[
		l(p,\Gamma) = 2b -2a - h(x) + h(y).
	\]
    \end{corollary}
	
    \begin{proof}
        Let us suppose that the points are as in case 2 in the proof of Proposition \ref{prop_inversions}, which is the more general case. Let us also assume that $h(x)\leq h(y)$, the proof for the other case is almost identical with the role of $x$ and $y$ reversed. Then $a=\omega_j$ and $b=\omega_{j+1}$. The associated path $(p,\Gamma)$ is the concatenation of three monotone paths:
	\begin{enumerate}[(1)]
		\item the downward going path connecting $x$ to $[x_1,a]$, whose length is $h(x) - a$,
		\item the upward going path connecting $[x_1,a]$ to $[y_1,b]$, whose length is $b - a$,
		\item the downward going path connecting $[y_1,b]$ to $y$, whose length is $b - h(y)$.
	\end{enumerate}
	By adding together we get $l(p,\Gamma) = 2b - 2a + h(x) - h(y)$ as requested.
    \end{proof}

    We are now ready to conclude this section.
	
    \begin{theorem}[Restatement of Proposition 1.1 in \cite{laakso2000ahlfors}]\label{maintheorem}
        Let $x,y\in\C{L}$ and let $[a,b]$ be a minimal interval for $x$ and $y$. Then a path associated to $[a,b]$ is a geodesic.
    \end{theorem}

    \begin{proof}
        The proof in the case when $x$ and $y$ can be connected by a monotone path is a consequence of Proposition \ref{prop_distancemonotone} and Corollary \ref{lemma:monotoneminimal}, so we can assume that the two points cannot be connected by such a path. Let us further assume that $h(x)\leq h(y)$ and the points are as in case 2 in the proof of Proposition \ref{prop_inversions}. The other case follows by taking $b=h(y)$ (or $a=h(x)$). 
  
        Let $(p,\Gamma)$ be a path associated to the minimal interval. Clearly
        \begin{equation}\label{eqdistlength}
	   2b-2a-h(y)+h(x)=l(p,\Gamma)\geq d(x,y).
        \end{equation}
        This comes from the definition of $d(x,y)$ as an infimum among the lengths of paths connecting $x$ and $y$ and from the fact that  $(p,\Gamma)$ is just one among those paths.
		
        Now take $(p',\Gamma')$ to be any another path that connects $x$ and $y$. Clearly it must start at $x$ and end at $y$. What we are really interested in, however, is what it does in between. In particular we are interested in the points in which it jumps using wormholes level of order $n_j$ and $n_{j+1}$. Let us call these points (notice that they are wormholes) $\theta_j$ and $\theta_{j+1}$ respectively. We first notice that $h(\theta_j),h(\theta_{j+1})\notin[h(x),h(y)]$, since we already established that no wormholes levels of order $n_j$ or $n_{j+1}$ are in $[h(x),h(y)]$. Let us suppose that $\theta_j$ appears before $\theta_{j+1}$ while travelling along $(p',\Gamma')$, the proof in the other case is similar. We divide $(p',\Gamma')$ in three part:
        \begin{enumerate}[(1)]
            \item $(p'_1,\Gamma'_1)$, connecting $x$ to $\theta_j$,
            \item $(p'_2,\Gamma'_2)$, connecting $\theta_j$ to $\theta_{j+1}$,
            \item $(p'_1,\Gamma'_1)$, connecting $\theta_{j+1}$ to $y$.
        \end{enumerate}
        While we do not know how all this paths are defined, what we do know is that $x$ can be connected to $\theta_j$ by a monotone path. The same is true for connecting $\theta_j$ to $\theta_{j+1}$ and $\theta_{j+1}$ to $y$. This follow easily from how $\theta_j$ and $\theta_{j+1}$ were defined. Hence can use \eqref{eq:inequalitydistance} to bound from below their length:      
        \[
            l(p'_1,\Gamma'_1)\geq |h(x)-h(\theta_j)|,\ l(p'_2,\Gamma'_2)\geq|h(\theta_{j+1})-h(\theta_j)|\ \text{and}\ l(p'_3,\Gamma'_3)\geq|h(\theta_{j+1})-h(y)|.
        \]
        Now recall that $a=\omega_j$ and $b=\omega_{j+1}$ and these wormhole levels were chosen in the proof of Proposition \ref{prop_inversions} to be the closest to the interval $[h(x),h(y)]$. Hence $|h(x)-h(\theta_j)|\geq h(x)-a$ and $|h(\theta_{j+1})-h(y)|\geq b-h(y)$. Moreover $|h(\theta_{j+1})-h(\theta_j)|$ is the distance between two wormholes of depth $n_j$ and $n_{j+1}$ respectively. Hence, since $[a,b]$ is a minimal interval, $|h(\theta_{j+1})-h(\theta_j)|\geq b-a$. By adding we get
        \[
            l(p',\Gamma')=l(p'_1,\Gamma'_1)+l(p'_2,\Gamma'_2)+l(p'_3,\Gamma'_3)\geq 2b - 2a -h(y) + h(x) = l(p,\Gamma).
        \]
        From this we conclude that no other path connecting $x$ to $y$ is shorter than $(p,\Gamma)$. Hence the $\geq$ in \eqref{eqdistlength} is actually an $=$ and this allows us to conclude that $l(p,\Gamma)$ is a geodesic.
    \end{proof}
	
	We conclude this section by noticing that \cite[Proposition 1.2]{laakso2000ahlfors} follows from this theorem and from Corollary \ref{coroll_length}, hence completing our goal of proving the results left unproved in the first part of \cite{laakso2000ahlfors}.

	\section{The special case $s=3$}
	To better understand the construction of $\C{L}$ and how geodesics work, we show what happens in a simple case. In particular we chose $s=3$, that will correspond to a Laakso space of dimension $Q=1+\log_3 2$. The IFS from equation \ref{eq_IFS} is then
	\[
		f_0(x) = \frac{x}{3}\ \  \text{ and }\ \  f_1(x) = \frac{x}{3} + \frac{2}{3}
	\]
	and the corresponding attractor is the classical middle-third Cantor set.

	The unique integer such that $n\leq s < n+1$ is clearly $n=s=3$. To see where the identifications are gonna happen in $\C{F}\times I$ we start by computing a sequence $\textbf{m}=\{m_i\}$. It must satisfy:
	\begin{enumerate}[(1)]
		\item $m_i\in\{3,4\}$,
		\item $$\frac{3}{4} \prod_{j=1}^i m_j^{-1} \leq \frac{1}{3^i}\leq \frac{4}{3}\prod_{j=1}^i m_j^{-1}.$$
	\end{enumerate}

	It is easy to see that the constant sequence $m_i=3\ \forall i$ satisfies the requirements. More in general, a sequence that satisfies all the requirements must be of the form $m_i = 3$ for each but up to one $i$, which can be either 3 or 4.

	Notice that a different choice of the sequence $\textbf{m}$ would result in a slightly different space. However, the main properties that we proved in the previous section are independent on $\textbf{m}$. Hence, for the example we want to study, we can choose the one that gives the nicest computations, i.e. $\textbf{m} = \{ 3 \}_i$. 

	With this choice the value of $\omega(n_1)$ is just $\frac{n_1}{3}$ for $n_1\in\{ 1,2 \}$, which means that the wormhole levels of order 1 are only $\frac{1}{3}$ and $\frac{2}{3}$. Notice that choosing the sequence with $m_1=4$ would give as wormhole levels of order 1 the values $\frac{1}{4},\,\frac{2}{4}$ and $\frac{3}{4}$.

	With a quick computation we have the wormhole levels of order 2. Recall that the second entry of $\omega_2(\cdot,\cdot)$ cannot be 0 and, in general, this is true for the $k$-th entry of $\omega_k$.
	\begin{table}[H]
		\centering
		\begin{tabular}{lll}
		$\omega(0,1) = \nicefrac{1}{9}$ & $\omega(0,2) =\nicefrac{2}{9}$ & $\omega(1,1) = \nicefrac{4}{9}$\\
		$\omega(1,2) = \nicefrac{5}{9}$ & $\omega(2,1) = \nicefrac{7}{9}$ & $\omega(2,2) = \nicefrac{8}{9}.$
		\end{tabular}
	\end{table}
	We also show wormhole levels of order 3:
	\begin{table}[H]
		\centering
		\begin{tabular}{lll}
		$\omega(0,0,1) = \nicefrac{1}{27}\ $& $\omega(0,0,2) = \nicefrac{2}{27}\ $& $\omega(0,1,1) = \nicefrac{4}{27}$\\
		$\omega(0,1,2) = \nicefrac{5}{27}\ $& $\omega(0,2,1) = \nicefrac{7}{27}\ $& $\omega(0,2,2) = \nicefrac{8}{27}$\\
		$\omega(1,0,1) = \nicefrac{10}{27}\ $& $\omega(1,0,2) = \nicefrac{11}{27}\ $& $\omega(1,1,1) = \nicefrac{13}{27}$\\
		$\omega(1,1,2) = \nicefrac{14}{27}\ $& $\omega(1,2,1) = \nicefrac{16}{27}\ $& $\omega(1,2,2) = \nicefrac{17}{27}$\\
		$\omega(2,0,1) = \nicefrac{19}{27}\ $& $\omega(2,0,2) = \nicefrac{20}{27}\ $& $\omega(2,1,1) = \nicefrac{22}{27}$\\
		$\omega(2,1,2) = \nicefrac{23}{27}\ $& $\omega(2,2,1) = \nicefrac{25}{27}\ $& $\omega(2,2,2) = \nicefrac{26}{27}$.
		\end{tabular}
	\end{table}
	Since we choose $\textbf{m}$ to be the constant sequence we can easily compute every value for the functions $\omega_k$. Indeed  \eqref{omega} simplifies to
	\begin{equation}\label{defomegak}
		\omega(n_1,\dots,n_k) = \sum_{i=1}^k \frac{n_i}{3^i}.
	\end{equation}

	Now we can see how the map $\pi$ works in this example. Two points in $\C{F}\times I$ are identified if they have the same height, which is a wormhole level of order $k$ for some $k$, and if their horizontal distance is $\frac{2}{3^k}$.
	
	From the wormhole levels of order $1$ we get that all the points in $\C{F}_0\times\{\frac{1}{3}\}$ are identified with their corresponding point in $\C{F}_1\times\{\frac{1}{3}\}$ and the same is true for points in $\C{F}_0\times\{\frac{2}{3}\}$ with points in $\C{F}_1\times\{\frac{2}{3}\}$.
	
	A wormhole level of order 2, for example $\frac{5}{9}$, identifies points in $\C{F}_{00}\times\{\frac{5}{9}\}$ with the corresponding point in $\C{F}_{01}\times\{\frac{5}{9}\}$ and points in $\C{F}_{10}\times\{\frac{5}{9}\}$ with the corresponding point in $\C{F}_{11}\times\{\frac{5}{9}\}$. However, it will not identify points in $\C{F}_{00}\times\{\frac{5}{9}\}$ with points in $\C{F}_{11}\times\{\frac{5}{9}\}$, because their distance is at least $\frac{7}{9}$, which is bigger than the distance of $\frac{2}{9}$ that wormholes level or order 2 allow to cover.

	We can now show an example of how paths in a Laakso space look like. Take $x=[0,\frac{1}{5}]$ and $y=[\frac{20}{27},\frac{1}{10}]$. It is easy to see that $0$ and $\frac{20}{27}$ correspond to the strings $a=\overline{0}=000\dots$ and $b=101\overline{0}=101000\dots$ in the Cantor set. Hence $x_1\parallel x_2$ and the algorithm in the proof of Proposition \ref{proppath} will stop after a finite number of iterations.
	\begin{enumerate}[(1)]
		\item Let $q_0=x$. The smallest $i$ such that $(x_1)_i\neq (x_2)_i$ is $i=1$.
		\item The wormhole level of order 1 closest to $y_1$ is $\omega^1=\omega_1(1) = \frac{1}{3}$.
		\item We define $\gamma_1(t):=\left(0,\frac{1}{5}+t(\frac{1}{3}-\frac{1}{5})\right)$. $\gamma_1$ connects $(0,\frac{1}{5})$ to $(0,\frac{1}{3})$ and since $\frac{1}{3}$ is a wormhole level of depth 1, $\pi$ will identify $(0,\frac{1}{3})$ with $(x_1',\frac{1}{3})=:q_1$, where $x_1'=\nu^1(x_1)=1\overline{0}=1000\dots$.
		\item Since $\nu^1(x_1) \neq x_2$ we restart by replacing $q_0$ with $q_1$ and $x_1$ with $x_1'$. In the second iteration we find $i_2=3$, choose $\omega^2=\omega_3(1,0,1)=\frac{10}{27}$ (it is not the only possible choice), define the path $\gamma_2(t):=\left(\frac{2}{3},\frac{1}{3}+t(\frac{10}{27}-\frac{1}{3})\right)$ and $q_2:=(x_1'',\frac{10}{27})$, where $x_1''=\nu^3(x_1')$. Notice that $x_1''=101\overline{0}=x_2$, hence the halting condition is met and we stop the algorithm.
	\end{enumerate}
	To finish the construction of the path we define $\gamma_3(t):=(\frac{20}{27},\frac{10}{27}+t(\frac{1}{10}-\frac{10}{27}))$, that joins $q_2$ with $(\frac{20}{27},\frac{1}{10})$. Notice that this last part of the path is going downward, as it is clear from the sign of the coefficient of $t$.
	We then use all the line segments in $\C{F}\times I$ to define
	\begin{equation*}
		\Gamma(t):=
		\begin{cases}
			\gamma_1(3t)& \text{if } t\in[0,\frac{1}{3}]\\
			\gamma_2(3t-1)& \text{if } t\in[\frac{1}{3},\frac{2}{3}]\\
			\gamma_3(3t-2)& \text{if } t\in[\frac{2}{3},1]
		\end{cases}
		=
		\begin{cases}
			(0, \frac{1}{5}+\frac{2t}{5}) & \text{if } t\in[0,\frac{1}{3}]\\
			(\frac{2}{3}, \frac{8}{27} + \frac{t}{9}) & \text{if } t\in[\frac{1}{3},\frac{2}{3}]\\
			(\frac{20}{27}, -\frac{73t}{90} + \frac{246}{270}) & \text{if } t\in[\frac{2}{3},1]	
		\end{cases}
	\end{equation*}
	and $p=\pi(\Gamma([0,1]))$ is a continuous path in $\C{L}$ that joins $x$ and $y$.
	
	It is easy to see that $l(p,\Gamma) = \frac{119}{270}$, however we might wonder if $p$ is a geodesic. In order to answer this question we first notice that $[\frac{1}{10},\frac{1}{3}]$ is a minimal interval for the points $x$ and $y$. Hence, from Corollary \ref{coroll_length} and Theorem \ref{maintheorem}, $d(x,y) = \frac{11}{30} < \frac{119}{270}$, i.e. $p$ is not a geodesic. In particular a geodesic can be constructed by following the steps used to construct $p$, but with a different choice of $\omega^2$ in the second iteration (any wormhole level of order $3$ between heights $\frac{1}{3}$ and $\frac{1}{10}$ will do). Notice that the difference between the length of $p$ and the length of a geodesic is exactly $\frac{2}{27}$, i.e. the extra distance travelled in order to use the wormhole level $\omega^2$ in step (3).

	For completeness let us also see an example of a path that comes from a family of infinitely many line segments in $\C{F}\times I$. Take $x=[0,0]$ and $y=[1,1]$. A minimal interval for these points is the whole unit segment $[0,1]$, hence $d(x,y)=1$. Let us construct a path that connects the two points. Since $x_1=\overline{0}$ and $x_2=\overline{1}$, we have that $x_1\nparallel x_2$. If we take the jumps in increasing order of depth then the first segment in $\C{F}\times I$ that we need to define is $\{\overline{0}\}\times[0,\frac{1}{3}]$. The final point of this segment is $(0,\frac{1}{3})$ which, since $\frac{1}{3}=\omega_1(1)$, is identified by $\pi$ with the point $(\frac{2}{3},\frac{1}{3})$. The second segment is then $\{\frac{2}{3}\}\times[\frac{1}{3},\frac{4}{9}]$, whose final point is identified with $(\frac{8}{9},\frac{4}{9})$ (because $\frac{4}{9}=\omega_2(1,1)$). More in general the $j$-th iteration of the algorithm will correspond to defining the segment
	\begin{equation*}\label{eq.segment}
		\left\{\sum_{i=1}^j\frac{2}{3^i}\right\}\times \left[\sum_{i=1}^j \frac{1}{3^i},\sum_{i=1}^{j+1} \frac{1}{3^i}\right]\subset\C{F}\times I.
	\end{equation*}
	It is easy to verify that the end points of the segments are converging to the point $(1,\frac{1}{2})$, where $\frac{1}{2}$ is $\overline{\omega}$ in the proof of Proposition \ref{proppath}. The last segment we define is then $\{1\}\times[\frac{1}{2},1]$. Notice that the family $\Gamma$ that collects all the segments is countable. We are left to check the length of the path $p$ that connects $x$ to $y$ and is defined as $p:=\pi(\Gamma)$. It is easy to see that
	\[
	l(p,\Gamma) = \sum_{i\in\B{N}}\frac{1}{3^i} + \frac{1}{2} = 1,
	\]
	hence $p$ is a geodesic. 
	
\section{Further developments}

Research involving Laakso spaces can go in several directions. We list some possible ideas in no particular order.
\begin{enumerate}
    \item In \cite[Remark 3.2]{laakso2000ahlfors} Laakso hinted at the possibility of constructing variations of $\C{L}$ with other self-similar fractals instead of $\C{F}$ or using a unit cube $I^n$ instead of $I$. While, for the reasons explained in the introduction, it is not clear what he exactly had in mind, these speculations are still open for debate and the construction of Laakso spaces with different sets as \textquotedblleft base\textquotedblright\ could prove interesting.
    \item Related to the above point, it is worth mentioning \cite[Remark 7.8]{antonelli2023carnot}. There, the authors suggest that it would be interesting to investigate the construction of a Laakso-type space based on a non-Abelian Carnot model. This path is certainly worth exploring in the near future.
    \item In \cite{laakso2000ahlfors} and later in \cite{steinhurst2010diffusions} we can find many reasons, like the existence of upper gradients and the validity of a Poincar\'e inequality, that suggest the possibility of studying many classical results in the framework of Laakso spaces. This was again implied by Laakso in \cite[Remark 3.3]{laakso2000ahlfors}. As shown in \cite{capolli2022maximal}, some classical results that holds in many common settings like Euclidean spaces and Carnot groups does not hold in Laakso spaces. It could be interesting to further investigate in this direction and in particular to study to what degree the absence of a linear structure is an obstacle.
    \item The existence of UDS in Laakso spaces is still open. In \cite{capolli2022maximal} it was only proved that the classical approach does not work in $\C{L}$, but the existence of UDS is still on the table. It is possible that, given the particular nature of the setting, the construction of UDS (provided that they exist) takes a completely different approach.
    \item Laakso spaces can be constructed as inverse limit spaces, see \cite{cheeger2015inverse, lang2001bilipschitz}. An interesting question is then if techniques from \cite{capolli2022maximal} and this paper can be used to study geometric properties of more general metric measure spaces obtained as inverse limit.
\end{enumerate}

\medskip

{\bf Acknowledgments.} 
    The author is a member of {\it Gruppo Nazionale per l’Analisi Matematica, la Probabilità e le loro Applicazioni} (GNAMPA) of {\it Istituto Nazionale di Alta Matematica} (INdAM) and acknowledge the support of the group.

    The author also acknowledges the support granted by the European Union – NextGenerationEU Project “NewSRG - New directions in SubRiemannian Geometry” within the Program STARS@UNIPD 2021.

    Part of this work was done while the author was a post-doctoral stutend at the Polish Academy of Sciences (IMPAN) in Warsaw. The author is thankful to the institution for the support showed during this period.

    Lastly, the author is thankful to Andrea Pinamonti (University of Trento) and Gareth Speight (University of Cincinnati) for suggesting the topic and for useful moments of discussion and to Tomasz Adamowich (IMPAN) for encouraging him in the pursuit of this project.

    \bibliographystyle{plain}
    \bibliography{main}

\end{document}